\documentclass[12pt,english]{amsart}
\usepackage[a4paper]{geometry}
\usepackage{amssymb,amsmath}
\usepackage{enumerate}
\makeatletter

\usepackage{latexsym}
\usepackage{color}

\newtheorem{thm}{Theorem}[section]
\newtheorem{lemme}[thm]{Lemma}

\newtheorem{prop}[thm]{Proposition}

\newtheorem{coro}[thm]{Corollary}

\newtheorem*{thmA}{Theorem A}
\newtheorem*{thmB}{Theorem B}
\theoremstyle{remark}\newtheorem{rem}[thm]{Remark}

\def\B{{\mathbb{B}}}
\def\D{{\mathbb{D}}}

\def\R{{\mathbb{R}}}
\def\C{{\mathbb{C}}}

\def\N{{\mathbb{N}}}

\numberwithin{equation}{section} 
\numberwithin{figure}{section} 
\numberwithin{table}{section} 

\usepackage[bookmarksopen, naturalnames]{hyperref}

\makeatother

\usepackage{babel}

\begin{document}
	\title{Wild boundary behaviour of holomorphic functions in domains of $\mathbb{C}^N$}
	\author{S. Charpentier, \L . Kosi\'nski}
	\thanks{The first author is partly supported by the grant ANR-17-CE40-0021
		of the French National Research Agency ANR (project Front). The second author is partly supported by the NCN grant SONATA BIS no. 2017/26/E/ST1/00723}
	\address{Stéphane Charpentier, Institut de Mathematiques, UMR 7373, Aix-Marseille
		Universite, 39 rue F. Joliot Curie, 13453 Marseille Cedex 13, France}
	\email{stephane.charpentier.1@univ-amu.fr}
	\address{\L ukasz Kosi\'nski, Institute of Mathematics, Jagiellonian University, \L ojasiewicza 6, 30-348 Krak\'ow, Poland}
	\email{lukasz.kosinski@im.uj.edu.pl}
	
\begin{abstract}Given a domain of holomorphy $D$ in $\mathbb{C}^N$, $N\geq 2$, we show that the set of holomorphic functions in $D$ whose cluster sets along any finite length paths to the boundary of $D$ is maximal, is residual, densely lineable and spaceable in the space $\mathcal{O}(D)$ of holomorphic functions in $D$. Besides, if $D$ is a strictly pseudoconvex domain in $\mathbb{C}^N$, and if a suitable family of smooth curves $\gamma(x,r)$, $x\in bD$, $r\in [0,1)$, ending at a point of $bD$ is given, then we exhibit a spaceable, densely lineable and residual subset of $\mathcal{O}(D)$, every element $f$ of which satisfies the following property: For any measurable function $h$ on $bD$, there exists a sequence $(r_n)_n \in [0,1)$ tending to $1$, such that
	\[
	f\circ \gamma(x,r_n) \rightarrow h (x),\,n\rightarrow \infty,
	\]
for almost every $x$ in $bD$.
\end{abstract}
	
	\maketitle

\section{Introduction}

Let $\mathbb{B}_N=\{z\in \C^N,\,|z|<1\}$ denote the open unit ball of $\mathbb{C}^N$, $N\geq 2$, and $\D$ the open unit disc of the complex plane. If $D$ is a domain in $\C^N$, $N\geq 1$, $\mathcal{O}(D)$ will stand for the space of holomorphic functions endowed with the locally uniform convergence topology.

In 2000, Boivin, Gauthier and Paramonov proved that, given a domain $\Omega \subsetneq \R^N$ whose boundary has no component consisting of a single point and a homogeneous complex elliptic operator $L$ in $\R^N$, there exists an $L$-analytic function $f$ in $\Omega$ - \emph{i.e} a solution $f$ to $Lf=0$ on $\Omega$ - satysfing the following property: For any $\zeta$ in the boundary $b\Omega$ of $\Omega$ and any continuous path $\gamma:[0,1)\rightarrow \Omega$ with $\gamma(r)\rightarrow \zeta$ as $r\rightarrow 1$, the cluster set of $f$ along $\gamma$ is maximal (\emph{i.e.} equal to $\C \cup \{\infty\}$) \cite[Theorem 5]{BoiGauPar}. We recall that the \emph{cluster set 
of $f$ along $\gamma$} is defined as the set
\[
\left\{c\in \C\cup \{\infty\}:\, f\circ \gamma(r_n)\rightarrow c\text{ for some }(r_n)\text{ tending to }1\right\}.
\]
This result was completed and extended in \cite{BBCPCA} where it was proven that the set of such above functions contains a dense vector subspace of the space of $L$-analytic functions on $\Omega$ endowed with the locally uniform convergence topology. In particular, the previous results apply to $\mathcal{O}(D)$ where $D$ is a domain in the complex plane (with $L=\bar{\partial}$) and to the space of functions harmonic in domains of $\R^N$ (with $L=\Delta$), and thus improved earlier works, for instance in the unit disc. We refer to the survey \cite{Prado} for an overview on the topic of ($L$)-holomorphic functions with maximal cluster sets before 2008. We shall also mention \cite{BBCPRMI} where the authors are interested in the \emph{dense lineability} and the \emph{spaceability} (see below for the definition) in $\mathcal{O}(\D)$ of the set of so-called \emph{universal series} (see \cite{Nesto}) with maximal cluster sets along any path to the boundary, and \cite{charp} where functions holomorphic in $\D$ with a \emph{universal} property implying the maximality of cluster sets along any such paths are exhibited and studied more specifically.

Despite the rather large degree of generality of the previous results, it seems that the natural problem of exhibiting holomorphic functions in domains of $\C^N$, $N\geq 2$, with maximal cluster sets along continuous paths to the boundary has been left open till now. Actually polynomial and holomorphic approximation in $\C^N$ makes the problem more involved in this setting. Indeed, the existence of a function $f$ in $\mathcal{O}(D)$ with maximal cluster set along every path to the boundary clearly imposes on $D$ to be a domain of holomorphy, hence to be pseudoconvex by Oka's solution to the Levi problem. Moreover it turns out that such a function cannot exist in general: no holomorphic function in $\B_N$ can be unbounded along \emph{every} path to the boundary \cite[Theorem 3]{Glob-Stout}.

The first aim of this paper is to address this problem (Theorem A below). Let $D$ be a pseudoconvex domain in $\C^N$. From now on, a \emph{finite length path to $bD$} refers to as a continuous piecewise $\mathcal{C}^1$ path $\gamma:[0,1)\rightarrow D$, with finite length, such that $\gamma(r)\rightarrow \zeta \in bD$ as $r\rightarrow 1$. We shall also recall that a subset of a topological vector space $X$ is \emph{residual} if it contains a dense countable intersection of open sets, \emph{densely lineable} if it contains apart $0$ a dense subspace of $X$ and \emph{spaceable} if it contains apart $0$ a closed infinite dimensional subspace of $X$ (see \cite{Aron,Bernal}).

\begin{thmA}Let $D$ be a pseudoconvex domain in $\C^N$. There exists a set $\mathcal{V}$ residual in $\mathcal{O}(D)$ whose every element $f$ satisfies the following property: for any finite length path $\gamma$ to $bD$ and any complex number $c\in \C$, there exists a sequence $(r_n)_n\subset [0,1)$, $r_n\rightarrow 1$ as $n\rightarrow \infty$, such that
	\[
	f\circ\gamma (r_n)\rightarrow c ,\, n\rightarrow \infty.
	\]
Moreover $\mathcal{V}$ is densely lineable and spaceable in $\mathcal{O}(D)$.
\end{thmA}

Observe that, since $L\cap \B_N$ is the image of a finite length path for any complex line $L$, it implies that the set of those $f\in \mathcal{O}(\B_N)$ such that $f_{L\cap \B_N}$ is non-extendable for any $L$ is residual in $\mathcal{O}(D)$, an answer to a question posed in \cite[Page 8]{ACGMN}.

Let us make some comments on the proofs in \cite{BBCPCA,BoiGauPar}. The functions with maximal cluster sets along paths are built as (infinite) sums of $L$-analytic functions in $\Omega$ which are simultaneously small on some compact subsets of $\Omega$ and arbitrarily close to any given complex number on \emph{suitable} other compact subsets. Here \emph{suitable} means that any path to the boundary has to intersect all but finitely many of these sets. To perform this construction, it is built in \cite{BoiGauPar} a sequence $(G_n)$ of sets in $\Omega$ with the property that any path to $b\Omega$ intersects all but finitely many $G_n$'s and such that each $K_n \cup G_n$ is a Roth-Keldysh-Lavrentiev set - on which a Runge-type theorem for $L$-analytic function is valid - where $(K_n)$ is some exhaustion of $\Omega$ by compact sets (see \cite[Lemma 3.2]{BBCPCA} for a precise statement). In 2015, for completely different purpose, Globevnik succeeded in building a \emph{labyrinth}, that is a sequence $(\Gamma_n)$ in $\B_N$ with exactly the same property as the sequence $(G_n)$ above, except that only those paths to $b\B_N$ with \emph{finite length} have to intersect all but finitely many $\Gamma_n$ and that $K_n\cup \Gamma_n$ is now polynomially convex \cite{Glob-annals}. This major step allowed him to build a function $f\in \mathcal{O}(\B_N)$ unbounded along any finite length path to $b\B_N$, that he used to exhibit a complete, closed complex hypersurface in $\B_N$, and thus to give a positive answer to a question posed by P. Yang in 1977 \cite{Yang1,Yang2}. He extended this to any pseudoconvex domains of $\C^N$ \cite{Glob-math-annalen}. Later the results of Globevnik from \cite{Glob-annals,Glob-math-annalen} were proven by Alarc\'on \cite{Ala} for any Stein manifold endowed with a Riemannian metric.

The construction of such labyrinths was much simplified in $\B_N$ \cite{AlaGlobLop} and also in any pseudoconvex domains \cite{CharpKos}. With this in hand, we will then be able to prove Theorem A.

\medskip{}

Furthermore, it turns out that holomorphic functions in $\C^N$ with some erratic boundary behaviour were already exhibited in 2005 by Bayart \cite[Theorem 1]{Bay}, who showed the following.

\begin{thm}The set of functions in $\mathcal{O}(\B_N)$ satisfying that, given any measurable function $h$ on $b\B_N$, there exists an increasing sequence $(r_n)_n\subset [0,1)$, $r_n\rightarrow 1$, such that
	\[
	f(r_nx)\rightarrow h(x)\text{ for almost every }x \in b\B_N
	\]
is residual.
\end{thm}
In fact, a more difficult result is proven showing that such $f$ can have an arbitrarily slow radial growth to $b\B_N$.
This result is in contrast with the existence of holomorphic functions with radial limits equal almost everywhere on $b\B_N$ to a given measurable function (see \cite{KahKatz} for the unit disc, \cite{Dup,Iordan} for the unit ball, and \cite{CharpDupMoun} for strictly pseudoconvex domains in $\C^N$ with $\mathcal{C}^{\infty}$ boundary). In \cite{Bay,Iordan} the constructions are based on a Mergelyan type theorem for certain star shaped compact subsets of the closed ball, proven in \cite{HakSib}. 

The second aim of this paper is to prove a simpler and natural polynomial approximation result in strictly convex sets of $\C^N$ - somewhat inspired by the construction of labyrinths, to extend it to strictly pseudoconvex domains thanks to a deep embedding theorem due to Diederich, Forn\ae ss and Wold \cite{Died-For-Wold}, and finally to improve and extend, in a rather transparent fashion, Bayart's result to strictly pseudoconvex domains.

\begin{thmB}[Theorem \ref{thmBrestated} below for a precise statement]Let $D$ be a strictly pseudoconvex domain of $\C^N$ and let $\gamma: bD\times[0,1]\to \overline D$ be a smooth map. Under some assumptions on $\gamma$, there is a residual set $\mathcal{W}$ in $\mathcal{O}(D)$ every element $f$ of which satisfies that given any measurable function $h:bD\to \mathbb C$, there exists a sequence $(r_n)_n\subset [0,1)$, $r_n\rightarrow 1$, such that
	\[
	f\circ\gamma(x,r_{n}) \to h(x)
	\]
for almost every $x\in bD$. Moreover, the set $\mathcal{W}$ is densely lineable and spaceable in $\mathcal{O}(D)$.
\end{thmB}
If $D=\B_N$ our proof can be combined with that of Bayart (or Iordan \cite{Iordan}) to get that any functions as in Theorem B can have an arbitrary radial growth.


The paper is organized as follows. The next section introduces the necessary background in complex analysis in several variables, that readers from Operator Theory or one dimensional complex variable may not be familiar with. Section \ref{sec-thmA} is devoted to the proof of Theorem A, and Section \ref{sec-thmB} to that of Theorem B.


\section{Background}\label{background}

In this section, we briefly introduce the necessary background in several complex variables.

\smallskip{}

A domain $D$ in $\C^N$ is said to be \emph{pseudoconvex} if there is a continuous plurisubharmonic function $u:D \to \R$ such for any $c\in \R$ the set $\{z\in D:\,u(z)<c\}$ is relatively compact. We refer to \cite[Chapter II]{Horm} for several equivalent definitions. As a solution to the Levi problem, Oka, Bremermann and Norguet established that the class of pseudoconvex domains coincides with that of domains of holomorphy \cite[Chapter IV]{Horm}.

\smallskip{}

%


A \emph{strictly} (or \emph{strongly}) pseudoconvex domain $D$ in $\C^N$ is a bounded domain with $\mathcal{C}^2$ boundary such that the Levi form of some defining function $\rho$ is positive definite on the complex tangent space (see \cite[Chapter 3]{Krantz} or \cite[Chapter 3]{Ohsawa}). The next theorem will be crucial for the proof of Theorem B. It is a particular case of \cite[Theorem 1.1]{Died-For-Wold}.

\begin{thm}\label{thm-Died}Let $D\subset \C^N$ be a strictly pseudoconvex domain. For any $x\in bD$, there is a ball $B_x$ centred at $x$ and a holomorphic embedding $k_x:\overline D\to \overline \B_N$ with $k_x(x) = (1,0)\in \C\times \C^{N-1}$ such that each point of $k_x(B_x \cap bD)$ is a point of strict convexity of $k_x(D)$.
\end{thm}

\smallskip{}

We will constantly make use of polynomial or holomorphic approximation. This is understood in terms of polynomial or holomorphic convexity, that we shall recall now. For this discussion, we refer to the book by Stout \cite{Stout}.

Let $D$ be a domain of $\C^N$ and $K$ a compact subset of $D$. The \emph{holomorphic hull of $K$ with respect to $D$} is defined by
\[
\widehat K_D=\{z\in D:\ |\varphi(z)|\leq \sup_K|\varphi|\text{ for every } \varphi\in \mathcal O(D)\}.
\]
It is well known that $\widehat K_D$ is a compact subset of $D$ provided that $D$ is pseudoconvex. $K$ is said to be \emph{holomorphically convex with respect to $D$} - or \emph{$\mathcal{O}(D)$-convex} - if $K=\widehat K_{D}$. If $D = \mathbb C^N$, then $\widehat K=\widehat K_{\mathbb C^N}$ is the \emph{polynomial hull} of $K$, and $K$ is \emph{polynomially convex} whenever $\widehat K=K$. Any compact convex set is polynomially convex. The $\C^N$ version of Runge's theorem is the Oka-Weil theorem, that states as follows (see \cite[Theorem 2.3.1]{Forstneric}).

\begin{thm}Let $D$ be a domain of $\C^N$. If $K\subset D$ is a compact $\mathcal{O}(D)$-convex set, then any function holomorphic in a neighborhood of $K$ can be approximated by functions in $\mathcal{O}(D)$ uniformly on $K$.
\end{thm}

A union of polynomially convex sets is not polynomially convex in general. However, a classical result due to Kallin says that, under some conditions, this can hold true (see e.g. \cite[Theorem 1.6.19]{Stout}). In particular, it says that if $K$ and $L$ are compact subsets of $\mathbb C^N$ and there is a polynomial $p$ such that $\widehat{p(K)}\cap \widehat{p(L)}$ is empty, then $\widehat{K\cup L} = \widehat K \cup \widehat L$. Thus if $K$ and $L$ are additionally polynomially convex, then $K\cup L$ is polynomially convex as well. In the next sections we shall need the following analogous result for holomorphic hulls. It can be proven as \cite[Theorem 1.6.19]{Stout} or by embedding $D$ into $\C^{2N+1}$ (\cite[Theorem 5.3.9]{Horm}), using Cartan's theorem and directly applying \cite[Theorem 1.6.19]{Stout}.

\begin{prop}\label{Kallin-hol}
Let $K,L$ be compact subsets of a pseudoconvex domain $D$ in $\mathbb C^N$. Suppose that there is $f\in \mathcal O(D)$ such that $\widehat{f(K)}\cap \widehat{f(L)}$ is empty. Then $$\widehat{K\cup L}_D = \widehat K_D \cup \widehat L_D.$$	
\end{prop}

This proposition and an easy induction immediately yields the next corollary.

\begin{coro}\label{cor:pol} Let $D\subset \mathbb C^N$ be a pseudoconvex domain. Let $K_0,\ldots, K_n$ be compact $\mathcal{O}(D)$-convex sets. Suppose that for each $j=0,\ldots,n$ there is $f\in \mathcal O(D)$ such that $\widehat{f(K_j)}_D\cap \widehat{f(\bigcup_{i<j} K_i)}_D$ is empty. Then $K_0\cup\ldots \cup K_n$ is $\mathcal{O}(D)$-convex.

\end{coro}


\section{Proofs of Theorems A and B and some complementary results}

\subsection{Wild holomorphic functions along finite length paths}\label{sec-thmA}

Given an open set $U\subset \C^N$, a map $\gamma :[0,1)\to U$ is said to be a path to $bU$ - or ending at $bU$ - if $\gamma$ is piecewise $\mathcal{C}^1$ and there exists $x\in bU$ such that $\lim_{r\to 1} \gamma(r)=x$. The length of a path $\gamma$ is defined as $\int_0^1|\gamma'(t)|dt$. The main goal of this section is to prove Theorem A. Let us first prove the first part, that we restate as follows.

\begin{thm}\label{theo-finite-path-pseudoconvex}
	Let $D$ be a pseudoconvex domain of $\mathbb C^N$. Let $\mathcal{V}$ denote the subset of $\mathcal{O}(D)$ consisting in those $f$ satisfying the property that given any finite length path $\gamma$ to $bD$ and any complex number $c$, there exists $r_n$ tending to $1$ such that $f\circ \gamma(r_n)\rightarrow c$ as $n\rightarrow \infty$. Then $\mathcal{V}$ is residual in $\mathcal{O}(D)$.
\end{thm}
	
The proof will be an easy application of the Baire Category Theorem and of the main result of \cite{CharpKos} that we state below under the form we need.

\begin{thm}\label{thm-CharpKos}
	Let $D$ be a pseudoconvex domain in $\C^N$ and let $(D_n)$ be an exhaustion of $D$ by strictly pseudoconvex domains that are $\mathcal O(D)$-convex. Let also $(M_n)$ be a sequence of positive numbers. Then there are compact sets $\Gamma_n\subset D_{n+1}\setminus\overline D_{n}$ such that $\overline D_n\cup\bigcup_{j=n}^m \Gamma_j$ is $\mathcal O(D)$-convex for every $m\geq n$, and any continuous path connecting $\partial D_n$ to $\partial D_{n+1}$ and omitting $\Gamma_n$ has length greater than $M_n$.
\end{thm}

\begin{proof}[Proof of Theorem \ref{theo-finite-path-pseudoconvex}]Let $(K_n)$ be given by Theorem \ref{thm-CharpKos} with $M_n=1$ for any $n$, and let $(c_j)_j$ be a dense sequence in $\mathbb{C}$. We define
	\[
	\tilde{\mathcal{V}}:=\bigcap_{k,j\in \N} \bigcup_{n\in \N}\left\{f\in \mathcal{O}(D):\,\sup_{\Gamma_n}|f-c_j|<\frac{1}{k}\right\}.
	\]
	We claim that any $f\in \tilde{\mathcal{V}}$ belongs to $\mathcal{V}$. Indeed, for a given $c\in \mathbb{C}$, there exists a sequence $(n_k)_k$ such that $\sup_{\Gamma_{n_k}}|f-c|<\frac{1}{k}$. Let now $\gamma$ be any continuous path with finite length ending at $b D$. By the properties of the $\Gamma_n$'s, $\gamma([0,1))$ must intersect all but finitely many $\Gamma _{n_k}$. Hence $c$ belongs to the closure of $f\circ \gamma ([0,1))$.
	
	It remains to prove that $\tilde{\mathcal{V}}$ is a dense $G_{\delta}$-subset of $\mathcal{O}(D)$. By the Baire Category Theorem, it is enough to prove that for any $k,j$, the set $\bigcup_{n\in \N}\left\{f\in \mathcal{O}(D):\,\sup_{\Gamma_n}|f-c_j|<\frac{1}{k}\right\}$ is open and dense in $\mathcal{O}(D)$. That it is open is clear. To prove it is dense, we fix a compact subset $K\subset D$, $h\in \mathcal{O}(D)$ and $\epsilon>0$. Then we choose $n$ such that $K\subset D_{n}$ and use the $\mathcal{O}(D)$-convexity of $\overline D_{n} \cup \Gamma_n$ to get $f\in \mathcal{O}(D)$ such that $|f-h|<\epsilon$ on $\overline D_{n}$ and $|f-c_j|<1/k$ on $\Gamma_n$
	
\end{proof}

Let us now complete the proof of Theorem A.

\begin{prop}\label{lineability-V}With the notations of Theorem \ref{theo-finite-path-pseudoconvex}, $\mathcal{V}$ is dense lineable and spaceable in $\mathcal{O}(D)$.
\end{prop}

\begin{proof}(1) We first prove the dense lineability of $\mathcal{V}$. Let us fix a countable neighbourhood basis $(U_l)$ of $\mathcal{O}(D)$. With the notations used in the proof of Theorem \ref{theo-finite-path-pseudoconvex}, we easily derive from the latter that, for any increasing sequence $(n_i)$ of integers, the set
	\[
	\tilde{\mathcal{V}}((n_i)):=\bigcap_{k,j\in \N} \bigcup_{i\in \N}\left\{f\in \mathcal{O}(D);\,\sup_{\Gamma_{n_i}}|f-c_j|<\frac{1}{k}\right\}.
	\]
is residual in $\mathcal{O}(D)$. Let us build by induction sequences $(n_i^l)$ in $\mathbb{N}$ and $(f_l)$ in $\tilde{\mathcal{V}}$ as follows: Pick $f_0 \in \tilde{\mathcal{V}}\cap U_0$ and consider an increasing sequence $(n_i^0)$ such that $\sup_{\Gamma_{n_i^0}}|f_0|\to 0$ as $i\to \infty$. Now assume that $f_l$ has been built in $\tilde{\mathcal{V}}((n_i^{l-1}))\cap U_l$ (with the notation $(n_i^{-1}):=(n_i)$). There is an increasing subsequence $(n_i^l)$ of $(n_i^{l-1})$ for which $\sup_{\Gamma_{n_i^l}}|f_l|\to 0$ as $i\to \infty$. Let us pick any $f_{l+1}\in \tilde{\mathcal{V}}((n_i^{l}))\cap U_{l+1}$. Thus the set $E:=\text{span}(f_l:\,l\geq 0)$ is a dense vector subspace of $\mathcal{O}(D)$.

To finish, let us check that any $f:=\sum _{l=0}^Ma_lf_l\in E\setminus\{0\}$ belong to $\tilde{\mathcal{V}}$. We can assume that $a_M\neq 0$. Fix $k,j\in \N$. By definition, there exist infinitely many $n_i^{M}$ such that
\[
\sup_{\Gamma_{n_i^{M}}}|f_M-\frac{c_j}{a_M}|<\frac{1}{2k}
\]
Moreover, by construction, $\sup_{\Gamma_{n_i^{M}}}|f_l|\to 0$ as $i\to \infty$, for any $0\leq l \leq M-1$. Thus one can choose $i\in \N$ such that $\sup_{\Gamma_{n_i^{M}}}|f-c_j|<\frac{1}{k}$, which proves that $f\in \tilde{\mathcal{V}}((n_i^M))$. Since we clearly have $\tilde{\mathcal{V}}((n_i^M))\subset \tilde{\mathcal{V}}$, the proof is complete.

\medskip{}

\noindent(2) We turn to the proof of the spaceability of $\mathcal{V}$. We denote by $\prec$ the lexicographical order on $\N \times \N$. Let $(D_n)$ and $(\Gamma_n)$ be given as in Theorem \ref{thm-CharpKos}. Up to re-ordering we may assume that the $\Gamma_n$'s are ordered in a double sequence $\Gamma_{k,n}$, $k\geq n\geq 1$, such that $\Gamma_{k,n}\subset D_{j+1}\setminus \overline{D_j}$ where $j$ is the rank in which the pair $(k,n)$ appears in the sequence $\N \times \N$ ordered increasingly with respect to $\prec$ (for instance, if $(k,n)=(4,2)$ then its rank $j$ is $8$). We also fix a basic sequence $(e_n)$ in $\mathcal{O}(D)$, generating an infinite dimensional subspace, such that for every $n\geq 1$, $\sup_{\overline{D_1}} |e_n| =1$ and the sequence $(e_k)_{k\geq n}$ is basic in the normed vector space $(\mathcal{O}(D),\sup_{\overline{D_n}}|\cdot|)$ with basic constant less than $2$. For the existence of such a basic sequence, see \cite[Lemma 1.7]{Menet}. Let $(c_k)$ be a dense sequence in $\C$ and let us fix a sequence $(\epsilon_n)$ of positive real numbers such that $\epsilon:=\sum \epsilon _n<\infty$.

As in the proof of \cite[Theorem 1.1]{Menet} and using the $\mathcal{O}(D)$-convexity of $\Gamma_n \cup \overline{D_n}$, one can build by induction two families $(f_{i,j})_{i\geq j\geq 1} \subset \mathcal{O}(D)$ and $(n_{i,j})_{i\geq j\geq 1} \subset \N$ satisfying the following conditions:
\begin{enumerate}[(i)]
	\item $\sup_{\overline{D_j}}|f_{j,j}|<\frac{\epsilon_j}{2^{j+1}}$
	;
	\item $\sup_{\Gamma_{k,j}}|f_{i,j}-(c_k-\sum_{l=j}^{i-1}f_{l,j})|<\frac{\epsilon_j}{2^{i+1}}$ if $\Gamma_{k,j} \subset D_{i+1}\setminus \overline{D_i}$;
	\item $\sup_{\overline{D_i}}|f_{i,j}|<\frac{\epsilon_j}{2^{i+1}}$;
	\item $\sup_{\Gamma_{k,m}}|f_{i,j}+\sum_{l=j}^{i-1}f_{l,j}|<\frac{\epsilon_j}{2^{i+1}}$ if $m\neq j$ and $\Gamma_{k,m}\subset D_{i+1}\setminus \overline{D_i}$.
\end{enumerate}

By (iii) $f_j:=\sum_{i\geq j}f_{i,j}+e_j$ is well-defined in $\mathcal{O}(D)$ and, according to \cite[Lemma 1.9]{Menet}, (iii) also implies that $(f_j)$ is a basic sequence in $\mathcal{O}(D)$, equivalent to $(e_j)$. Thus the set $E:=\overline{\text{span}}(f_j;\,j\geq 1)$ is a closed infinite dimensional subspace of $\mathcal{O}(D)$. The remaining of the proof consists in showing that any $f=\sum_ {n\geq 1}a_nf_n \in E \setminus \{0\}$ belongs to $\mathcal{V}$. Denoting by $N$ the smallest integer among those $n$ for which $a_n\neq 0$, we shall assume, without loss of generality, that $a_N=1$. Moreover, proceeding as in the proof of \cite[Theorem 1.10]{Menet}, one can see that there exists a constant $K$ such that $|a_n|\leq K$ for any $n\geq 1$.

By the proof of Theorem \ref{theo-finite-path-pseudoconvex}, we need only prove that $\sup_{\Gamma_{k,N}}|f-c_k|\rightarrow 0$ as $k\rightarrow \infty$. For $k\geq N$ let us denote by $j(k)$ the unique integer such that $\Gamma_{k,N}\subset D_{j(k)+1}\setminus \overline{D_{j(k)}}$. Note that $j(k)>k$. Considering the estimate $|f(z)-c_k| \leq |f_N(z)-c_k|+K\sum_{n\neq N}|f_n(z)|$, we use first (ii) and (iii) to get, for any $z\in \Gamma_{k,N}$,
\[
|f_N(z)-c_k|\leq |f_N(z)-\sum_{i=N}^{j(k)}f_{i,N}(z)|+|\sum _{i=N}^{j(k)}f_{i,N}(z)-c_k| < \frac{\epsilon_N}{2^{j(k)}},
\]
and secondly (iii) and (iv) to obtain
\[
\sum_{n\neq N}|f_n(z)|\leq \sum_{n\neq N}|\sum_{i=n}^{j(k)}f_{i,n}(z)|+\sum_{n\neq N}|f_n(z)-\sum_{i=n}^{j(k)}f_{i,n}(z)| < \frac{\epsilon}{2^{j(k)}}.
\]
This concludes the proof.
\end{proof}


\subsection{Almost everywhere wild holomorphic functions}\label{sec-thmB}

From now on, we fix a strictly pseudoconvex domain $D$ in $\C^N$, $\rho$ a defining function of $D$, and we denote by $m=m_{bD}$ the normalized measure on $bD$. For simplicity, we also denote by $m$ the Lebesgue measure in $\mathbb{R}^n$. In what follows we shall say that a smooth domain (\textit{i.e.} a domain with $\mathcal{C}^2$ boundary) $D'$ is \emph{close} to $D$ if it possesses a defining function $\rho'$ that is close to $\rho$ in $\mathcal C^2$ topology. Note that if $D'$ is sufficiently close to $D$, then it is also strictly pseudoconvex.

Let us consider a $\mathcal{C}^1$ map
$\gamma:bD\times[0,1]\to \overline D$ such that $\gamma(x,r)\in D$, $r\in [0,1)$, and  $r\mapsto\gamma(x,r)$ hits $bD$ transversally at $x$ for any $x\in bD$, as $r\to 1$.
For $r\in [0,1)$, we denote by $bD_r:=\{\gamma(x,r): x\in bD\}$. We assume additionally ($\dag$):

\smallskip{}

($\dag$)  {\it For any $r$ close enough to $1$, $bD_r$ bounds a smooth domain $D_r$ that is close to $D$.}

\smallskip{}

\noindent{}In particular, this assumption implies that for any compact set $K$ in $D$ and any $r$ close enough to $1$, $D_r$ contains $K$.

Note that given a strictly pseudoconvex domain $D$, there always exists a $\mathcal C^1$ map as above satisfying ($\dag$). Indeed, it is the most natural one: For $x\in bD$, let $\nu_x$ denote the inward unit normal vector and set $\gamma(x,r):= x + (1-r)\nu_x$ for $r\in[0,1]$. It is clear that $\gamma$ is $\mathcal C^1$ on $bD\times [0,1]$. Moreover, for $r$ close enough to $1$, each point of the set $bD_r$ is at distance $r$ from $bD$. Now, we can use the well-known fact that the signed distance function to the boundary of a smooth domain is $\mathcal{C}^2$ in order to see that $D_r$ is a smooth domain, which is close in $\mathcal C^2$-topology to $D$ (see \cite[Chapter 3]{Krantz}).

In fact, a classical but a bit more tricky argument shows that there even exist maps $\gamma$ which are $\mathcal{C}^2$ on $bD\times [0,1]$ and hit transversally $bD$ as $r\to 1$. The following proposition says that, under this stronger assumption, ($\dag$) is automatically satisfied.
\begin{prop}\label{lemgammapseudo} Keeping the above settings, let us assume additionally that the map $\gamma$ is $\mathcal C^2$ on $bD \times [0,1]$. Let $K$ be any compact subset of $D$. There exists $r_0 \in (0,1)$ such that $D_r$ is the boundary of a strictly pseudoconvex domain in $D$ that contains $K$ whenever $r\in (r_0,1].$
\end{prop}

\begin{proof}

Fix $x\in bD$.
It follows from the implicit function theorem that there exist a neighborhood $U$ of $x$, an open set $U'$ in $\mathbb R^{2n-1}$ and a $\mathcal C^2$ function $f:U \to \mathbb R$ such that
\[
U\cap bD=\{(t, f(t)):\ t\in U'\}.
\]
Now, using that $\gamma$ is $\mathcal{C}^2$, we shall infer that for any $r$ close enough to $1$, there are $\mathcal C^2$ smooth functions $g_r$ such that $\gamma((t, f(t)),r) = (t, g_r(t))$, $t\in U'$, and $g_r$ is close to $f$ in $\mathcal{C}^2$ topology. Thus $bD_r$ is locally the graph of a smooth function, from which we easily obtain, at first, that $bD_r$ bounds a domain, denote it by $D_r$, that contains $K$ for $r$ sufficiently close to 1. Second, we deduce that $D_r$ is strictly pseudoconvex, since the Levi form of some defining function for $D_r$, restricted to proper complex tangent spaces, is close to the (positive and definite) Levi form of a defining function of $D$.
\end{proof}

Let us now turn to the proof of Theorem B. The proofs of the dense lineability and the spaceability are omitted as both may be obtained exactly as Proposition \ref{lineability-V}, together with the proof of the following. Here and below, the map $\gamma$ is defined as above and assumed to be $\mathcal{C}^1$ and satisfy ($\dag$).

\begin{thm}\label{thmBrestated}
	Let $(r_j)$ be an increasing sequence, $0<r_j<1$, $\lim_{j\to \infty} r_j=1$. Then the set $\mathcal{W}(r_j)$ of holomorphic functions $f$ on $D$ such that, for any measurable function $h:bD\to \mathbb C$, there is a subsequence $(r_{j_k}) $ of $(r_j)$ with $$\lim_{k\to \infty} f(\gamma(x,r_{j_k})) = h(x)\quad  \text{for a.e.}\ x\in bD$$ is residual in $\mathcal O(D).$
\end{thm}

The proof of Theorem \ref{thmBrestated} is based on Lemma \ref{approx-lemma-imp} below about holomorphic approximation on $D$. To obtain this lemma, we will first prove a polynomial approximation result in strictly convex domains (Lemma \ref{lem:cups}), and then use embeddings of $D$ into strictly convex domains (Theorem \ref{thm-Died}).

For a convex set $V$ of $\C^N$ and $H$ a real hyperplane in $\C^{N}$, we call a \emph{cap} of $V$ any connected component of the set $V \cap (\C^N\setminus H)$ (which has at most two components). A closed cap of $V$ will be the closure in $\overline{V}$ of a cap of $V$. Note that any closed cap is a convex compact subset of $\mathbb{C}^N$ and, as such, it is polynomially convex.

\begin{lemme}\label{lem:cups}
	Let $V$ be a strictly convex domain and $U$ a domain of $\C^N$ intersecting $bV$. Then for every $\epsilon >0$ there are finitely many disjoint closed caps $\Gamma_1, \ldots, \Gamma_n$ contained entirely in $U\cap \overline{V}$ such that $$m((U \cap bV)\setminus \Gamma)<\epsilon,$$ where $ \Gamma=\Gamma_1\cup\ldots \cup \Gamma_n$. Moreover, given any polynomially convex compact subset $K$ of $V$ with $\Gamma \cap K = \emptyset$, $\Gamma \cup K$ is polynomially convex.
\end{lemme}


\begin{proof}
 
 
 By compactness and the implicit function theorem, we can assume that there exists an open set $U'$
 in $\mathbb{R}^{2N-1} $ such that 
 \[
 bV\cap U = \{(t, f(t)):\ t\in U'\},
 \]
 and that the map $F:t \mapsto (t,f(t))$ is a $\mathcal C^2$ diffeomorphism from $\overline{U'}$ onto $bV\cap \overline{U}$.
 Now it is not difficult to check that the strict convexity of $V$ and the smoothness of $F$ imply the existence of $0<K_1\leq K_2<\infty$ such that for each open cube $C$ contained in $U'$, there exists a closed cap $\Gamma$ of $V$ such that $F(C)\supset \Gamma\cap bV$ and $K_1m(C)\leq m(\Gamma)\leq K_2 m(C)$. Moreover, for $C$ small enough, we shall have $\Gamma \subset U$. Thus the lemma follows from the choice of finitely many disjoint open cubes $C_1,\ldots ,C_n$ in $U'$ such that the measure of $U'\setminus (\bigcup_{j=1}^n C_j)$ is small enough and, again, the fact that $F$ is a diffeomorphism.
 
 The last assertion of the lemma follows from Corollary \ref{cor:pol}.
\end{proof}

\medskip{}

\begin{lemme}\label{approx-lemma-imp}
 Let $D$ be a strictly pseudoconvex domain in $\C^N$. For any $\epsilon>0$, any compact $\mathcal O(D)$-convex set $K\subset D$, any strictly pseudoconvex domain $G\subset D$ sufficiently close to $D$ in $\mathcal{C}^2$-topology and any open set $U$ intersecting $bG$, there exists a compact set $\Gamma$ in $U\cap\overline G$ such that $m((U\cap bG)\setminus \Gamma)<\epsilon$, $K\cap \Gamma = \emptyset$ and $\Gamma \cup K$ is $\mathcal O(D)$-convex in $D$. 
\end{lemme}

\begin{proof}
 Up to obvious changes, we may and shall assume that $U=\mathbb{C}^N$. According to \cite{Died-For-Wold} (Theorem \ref{thm-Died}), for each $x\in bD$ there is a ball $B_x$ centred at $x$ and an embedding $k_x:\overline D\to \overline \B_N$ with $k_x(x) = (1,0)\in \C\times \C^{N-1}$ such that each point of $k_x(B_x \cap bD)$ is a point of strict convexity of $k_x(D)$. If $G$ is close enough to $D$, then
 each point of $k_x(B_x \cap bG)$ is a point of strict convexity of $k_x(G)$. Thus one can consider a strictly convex domain $V_x$ containing $k_x(G)$ such that $k_x(G) \cap k_x(B_x)= V_x \cap k_x (B_x)$. Moreover, we shall assume that $bG \subset \cup _{x\in bD}B_x$, hence one can find finitely many disjoint open sets $U_1,\ldots ,U_n$ such that each of them lies in some $B_x$, $K\cap \overline{U_j} =\emptyset$, and $m(bG \setminus \cup U_j) <\epsilon/2$.
 
 Let us fix $j$ and take $x$ such that $U_j\subset B_x$ and let us apply Lemma~\ref{lem:cups} to $V_x$ and the domain $k_x(U_j)$. It provides us with a finite union of disjoint closed caps, that we shall denote by $\tilde \Gamma_j$, such that $m((k_x(U_j)\cap b V_x)\setminus\tilde \Gamma_j)<\eta$, where $\eta$ will be chosen below. Let us define $\Gamma_j:=k_x^{-1}(\tilde \Gamma_j)$ and observe that $\eta$ can be chosen small enough so that
 \[
 m((U_j\cap b G)\setminus \Gamma_j)<\epsilon/(2n).
 \]
 Setting $\Gamma =\cup _j \Gamma_j$, we get $m(bG\setminus \Gamma)<\epsilon$.
 
 Note also that each $\tilde \Gamma_j$ is polynomially convex (Lemma \ref{lem:cups}) and is separated from 
 $k_x(K)$ and $k_x(U_i)$ for $i\neq j$ by a polynomial. Composing with $k_x$ we get that each $\Gamma_j$ is $\mathcal{O}(D)$-convex and separated from $K$ and $\Gamma_i$, $i\neq j$, by a function holomorphic in $D$. This, according to Corollary~\ref{cor:pol}, means that $\Gamma \cup K$ is $\mathcal O(D)$ convex.
 
\end{proof}

We now deduce the following, which is the key-ingredient for Theorem \ref{thmBrestated}.

\begin{coro}\label{approx-result}Let $\epsilon>0$, let $K$ be $\mathcal O(D)$-convex in $D$, $h\in \mathcal{O}(D)$ and let $\varphi$ be continuous on $bD$. There exists $r_0 \in (0,1)$ such that for any $r\in (r_0,1)$, there exist a compact set $E\subset bD$ and $f\in \mathcal{O}(D)$ such that
	\begin{enumerate}
		\item $m(bD \setminus E) < \epsilon$;
		\item $\sup_{x\in E}|f\circ \gamma (x,r)-\varphi(x)|<\epsilon$;
		\item $\sup_{z\in K}|f(z)-h(z)|<\epsilon$.
	\end{enumerate}
\end{coro}

\begin{proof}We fix $\epsilon >0$. By continuity of $\varphi$ on $bD$, there exists finitely many connected open sets $U_1,\ldots ,U_n \subset \mathbb{C}^N$, with $\overline{U_i}\cap \overline{U_j}=\emptyset$ if $i\neq j$, and $c_1,\ldots,c_n\in \mathbb{C}$ such that $|\varphi(x)-c_i|<\epsilon/2$ for any $x\in bD\cap U_i$, $i=1,\ldots,n$, and $m(bD\setminus (U:=\cup _iU_i))<\epsilon/2$.

Since $\gamma$ is $\mathcal{C}^1$ on $bD\times [0,1]$ and $\gamma (x,1)=x$ for any $x\in bD$, there exists $r_0$ such that for any $r_0\leq r\leq 1$, the map $\gamma_r:x\mapsto \gamma(x,r)$ is a $\mathcal{C}^1$ diffeomorphism from $bD$ onto $bD_r$. Moreover for any $A\subset bD$, $m(A)\approx m(\gamma_r(A))$, where the constants in $\approx$ do not depend on $A$ or on $r\geq r_0$. We define $V_r$ as an open set of $\mathbb{C}^N$ such that $V_r\cap K=\emptyset$ and $V_r=\cup _i V_r^i$ with $V_r^i=\gamma_r(bD\cap U_i)$. In particular, we thus have $m(bD_r\setminus V_r)\approx \epsilon/2$, and the function $\tilde{\varphi}$ defined as $\tilde{\varphi}(z)=c_i$ if $z\in V_r^i$ is well-defined and holomorphic in $V_r$.
	
By assumption ($\dag$), we can assume that $r_0$ is even closer to $1$ so that for any $r\in (r_0,1)$, $bD_r$ is the boundary of a strictly pseudoconvex domain in $D$ that contains $K$, whence, according to Lemma \ref{approx-lemma-imp}, there exists a compact set $\Gamma_r$ in $V_r$ with $m(V_r\cap bD_r\setminus \Gamma_r)<\epsilon /2$, such that $\Gamma_r \cup K$ is $\mathcal{O}(D)$-convex in $D$.

Thus there exists $f\in \mathcal{O}(D)$ such that $|f(z)-\tilde{\varphi}(z)|<\epsilon/2$ for any $z\in bD_r\cap\Gamma_r$ and $|f(z)-h(z)|<\epsilon$ for any $z\in K$. Setting $E=\gamma_r^{-1}(bD_r\cap \Gamma_r)$ completes the proof, since $\epsilon$ is arbitrary.

\end{proof}

\begin{proof}[Proof of Theorem \ref{thmBrestated}]The rest of the proof is very similar to steps 2 and 3 of that of \cite[Theorem 1]{Bay}. Let $(K_s)$ be $\mathcal O(D)$ convex sets that exhaust $D$, i.e. $K_j$ is contained in the interior of $K_{s+1}$ and $\bigcup_s K_s =D$. We also fix a dense sequence $(h_j)$ in the set $\mathcal{C}(bD)$ of continuous functions on $bD$. We denote by $\mathcal{E}_k$ the set of all compact subsets $E$ of $bD$ such that $m(bD\setminus E)>1-1/2^k$. For positive integers $j,k,l,s$, we define the set
	\[
	U(j,k,l,s)=\bigcup_{r\in(r_j,1)}\bigcup _{E\in \mathcal{E}_k}\left\{f\in \mathcal{O}(D):\,|f\circ \gamma (x,r)-h_j(x)|<1/2^k,\,x\in E\right\}.
	\]
It is clear that each $U(j,k,l,s)$ is an open subset of $\mathcal{O}(D)$ and, by Corollary \ref{approx-result}, that it is dense. So, by the Baire Category Theorem, we get that $U:=\bigcap_{j,k,l,s}U_{j,k,l,s}$ is a dense $G_{\delta}$-subset of $\mathcal{O}(D)$.

Let us take $f\in U$. It remains to prove that $f$ belongs to $\mathcal{W}(r_j)$. We thus fix $h$ measurable on $bD$. Proceeding as in Step 3 of the proof of \cite[Theorem 1]{Bay}, we can build by induction an increasing sequence $(r_{j_k})$, $r_{j_k}\to 1$, and a sequence $(E_k)$ in $\mathcal{E}_k$ such that
\[
|f\circ\gamma(x,r_{j_k})-h(x)|<\frac{1}{2^k},\quad x\in E_k.
\]
Setting $E=\bigcup_{n\geq 1}\bigcap_{k\geq n}E_k$, it is easily seen that $m(E)=1$ and that $f\circ\gamma(x,r_{j_k}) \to h(x)$ as $k\to \infty$ for any $x\in E$.
\end{proof}

\medskip{}

\subsubsection*{The case of the ball}In the case of the ball, Theorem \ref{thmBrestated} can be strengthened. Assume that we are given a family of maps $\{\gamma_z:b\mathbb{B}_N \times [0,1]\to \B_N;\,z\in \mathbb{B}_N\}$ such that each $\gamma_z$ satisfies the assumptions made at the beginning of the section and, additionally, that $(z,x,r)\mapsto \gamma_z(x,r)$ is continuous on $\mathbb{B}_N\times b\mathbb{B}_N \times [0,1]$ and $(z,r)\mapsto \partial \gamma_z(x,r)/\partial r$ is continuous on $\mathbb{B}_N\times [0,1]$, for any $x\in b\mathbb{B}_N$. A typical example of such maps is $\gamma_z(x,r)=x+r(x-z)$.


\smallskip{}

Let $v:[0,+\infty[\to [0,+\infty[$ be a weight, \emph{i.e.} an increasing function such that $v(x)\to \infty$ as $x\to \infty$. We recall that the space $H_v^0$ is defined as
\[
H_v^0:=\left\{f\in \mathcal{O}(\B_N):\,\sup_{|z|=r}\frac{|f(z)|}{v(r)}\rightarrow 0\text{ as }r\rightarrow 1\right\}.
\]
It is well-known that $H_v^0$ is a Banach space endowed with the norm $\Vert f \Vert = \sup _{z\in \mathbb{B}_N}|f(z)|/v(|z|)$ and that the polynomials are dense in $H_v^0$.

\begin{thm}\label{thm-strength-ball}Let $v$ be a weight and let $(r_j)$ be a sequence in $(0,1)$ with $r_j\to 1$. With the notations of Theorem \ref{thmBrestated} and $D=\B_N$, the set $\mathcal{W}(r_j)\cap H_v^0$
is residual in $H_v^0$.
\end{thm}

This theorem was obtained by Bayart for $\gamma_z(x,r)=z+r(x-z)$ \cite[Theorem 1]{Bay}. Observe that the statement holds true also with $\mathcal{O}(D)$ instead of $H_v^0$ and that it is already an improvement of Theorem \ref{thmBrestated} in the case $D=\B_N$ (the family $\gamma_z$ is indexed by an uncountable set and an uncountable intersection of residual sets may be empty). More surprisingly, the above theorem tells that the functions satisfying the conclusion of Theorem \ref{thmBrestated} can grow arbitrarily slowly. The proof of Theorem \ref{thm-strength-ball}, whose details are left to the reader, follows the same lines as that of \cite[Theorem 1]{Bay}, upon replacing \cite[Lemma 1]{Bay} and \cite[Lemma 3]{Bay} by, respectively, Lemmas \ref{lem:Bay-lemma1} and \ref{lemma-2-ball} below.

Let us recall that a continuous function $f$ on $\B_N$ has $K$-limit at $x\in \B_N$ if $(f(z_i))_i$ converges for any $(z_i)\subset D_{\alpha}(x)$ converging to $x$, where $D_{\alpha}(x)$ is the Kor\'anyi approach region defined for $\alpha>1$ by
\[
D_{\alpha}(x)=\left\{z\in \B_N:\,|1-\left<z,x\right>|<\frac{\alpha}{2}(1-|z|^2)\right\}.
\]
Observe that, under the assumptions made on the family $\{\gamma_z:\,z\in \B_N\}$, for any compact subset $L$ of $\B_N$, there exists $\alpha>1$ and $0<r<1$ such that
\begin{equation}\label{inclusion-utile}\left\{\gamma_z(x,r'):\,z\in L,\,r<r'<1\right\}\subset D_{\alpha}(\B_N).
\end{equation}
For $L$ a compact subset of $\B_N$ and $E\subset b\mathbb{B}_N$, we define
\[
E_L^*=\left\{\gamma_z(x,r);\,x\in E,\,r\in [0,1],\,z\in L\right\}
\]

\begin{lemme}\label{lem:Bay-lemma1}Let $L$ be a compact subset of $\B_N$ and $h$ a continuous function on $\mathbb{B}_N$ with $K$-limit almost everywhere on $b\mathbb{B}_N$. For every $\epsilon >0$, there exist a subset $E$ of $b\mathbb{B}_N$, $m(b\B_N\setminus E)<\epsilon$, and an extension $\tilde{h}$ of $h$ to $\mathbb{B}_N\cup E$ such that $\tilde{h}$ is continuous on $E_L^*$.
\end{lemme}

\begin{proof}Up to \eqref{inclusion-utile}, it is identical to that of \cite[Lemma 1]{Bay}.
\end{proof}

\begin{lemme}\label{lemma-2-ball}Let $v$ be a weight, $L$ a compact subset of $\B_N$, and $h$ a continuous function on $b\mathbb{B}_N$. For every $\epsilon>0$, there exist a compact subset $E$ of $b\mathbb{B}_N$ and a function $f$ holomorphic in $\mathbb{B}_N$, continuous on $E_L^*$, such that:
	\begin{enumerate}
		\item $m(\B_N \setminus E)<\epsilon$;
		\item $\vert f(x)-h(x) \vert<\epsilon$ for any $x\in E$;
		\item $\Vert f-g\Vert< \epsilon$.
	\end{enumerate}
\end{lemme}

\begin{proof}The proof works the same as that of \cite[Lemma 3]{Bay}. However, instead of \cite[Lemma 3]{Iordan} (see also \cite{HakSib}), which is the main ingredient of \cite[Lemma 3]{Bay}, we observe that one can use Lemma \ref{lem:cups} (with $V=\B_N$ and $U=\C^N$), whose proof is rather elementary.
\end{proof}

\begin{rem}It would be interesting to know whether Theorem \ref{thm-strength-ball} extends to any strictly pseudoconvex domain. The most difficult point would consist in controlling the growth of the functions near the boundary. The strategy developed in \cite{CharpDupMoun} might be useful.
\end{rem}

Proceeding as for $\mathcal{V}$ in Proposition \ref{lineability-V}, we can prove the following.

\begin{prop}Let $v$ be a weight and let $(r_j)$ be a sequence in $(0,1)$ with $r_j\to 1$. $\mathcal{W}(r_j)\cap H_v^0$ is densely lineable and spaceable in $H_v^0$. 
\end{prop}

\subsubsection*{A concluding remark}We recall that $\D$ stands for the unit disc in the complex plane. In the sequel we say that $f\in \mathcal{O}(\D)$ satisfies property ($\mathcal{P}$) if given any $K\subsetneq b\D$ and any $h\in \mathcal{C}(b\D)$, there exists $(r_j) \in [0,1)$ such that $f_{r_j}\to h$ uniformly on $K$ as $j\to \infty$. It is easily checked that if $f$ has ($\mathcal{P}$), then given any $h\in \mathcal{C}(b\D)$, $f_{r_j}$ converges pointwise on $b\D$ to $h$ for some sequence $(r_j)$, and that $f$ has maximal cluster sets along any path to $b\D$ (with finite or infinite length) $\gamma$ to $b\D$. It is proved in \cite{charp} that the set of functions satisfying ($\mathcal{P}$) is a dense $G_{\delta}$-subset of $\mathcal{O}(\D)$. In contrast, as we have already mentioned, there do not exist an $f\in \mathcal{O}(\B_N)$, $N\geq 2$, such that $f\circ \gamma$ has dense range for any path $\gamma$ to $b\B_N$ \cite{Glob-Stout}.

\medskip{}

Furthermore, one can show that there cannot exist $f\in \mathcal{O}(\B_N)$, $N\geq 2$, such that given any $h\in \mathcal{C}(b\B_N)$, $f_{r_j}\to h$ pointwise on $b\B_N$, for some $(r_j)\subset [0,1)$. Indeed, a standard argument using the Baire Category Theorem yields that if $f_{r_j}(x)\to h(x)$ for any $x\in b\B_N$, then it is uniformly bounded on some open subset of $b\B_N$. Thus, for some cap $\Gamma \in \B_N$, $f_{r_j}$ has to be bounded on $\Gamma \cap b\B_N$. Now, by the maximum modulus principle, $f_{r_j|L}$ is also bounded on $L\cap \Gamma$ uniformly for any complex affine line $L$ intersecting $\Gamma$. Therefore $f_{r_j}$ must be bounded on $\Gamma$, hence $f$ itself has to be bounded on a non-empty open set $\B_N \cap B_x$, for some open ball $B_x$ centred at some $x\in b\B_N$. This of course implies that for some $h\in \mathcal{C}(b\B_N)$, there exists no sequence $(r_j)\subset [0,1)$ such that $f_{r_j}\to h$ pointwise on $b\B_N$, a contradiction. Observe that this argument shows that the pointwise convergence of $f_{r_j}$ to any $h$ for some $(r_j)$ cannot even hold on a given \emph{open subset} of $b\B_N$.

\medskip{}

However, it is possible to prove the following. For $x\in b\B_N$, we denote by $L_x$ the complex line passing through $x$ and the origin, and by $bL_x$ the set $L_x\cap b\B_N$.

\begin{thm}\label{last-thm-sim}Let $(x_i)_{i\geq 1}$ be a sequence in $b\B_N$. There exists a dense $G_{\delta}$-subset $\mathcal{U}$ of $\mathcal{O}(\B_N)$ such that every $f\in \mathcal{U}$ satisfies that, given any $h_i\in \mathcal{C}(bL_{x_i})$ and any compact set $K_i\subsetneq bL_{x_i}$, $i\geq 1$, there exists a sequence $(r_j)$ in $[0,1)$ such that for any $i\geq 1$, $f_{r_j}\to h_i$ uniformly on $K_i$.
\end{thm}

Observe that the sequence $(r_j)$ does not depend on $i\geq 1$. For the proof, we will need another particular case of Kallin's lemma (see \cite[Page 63]{Stout}).

\begin{lemme}\label{lemma-Kallin-last}Let $K$ and $L$ be two compact convex subsets of $\C^N$ that intersects at one point. If there exists a linear functional $\varphi$ on $\C^N$ such that $\Re{\varphi}\leq 0$ on $K$ and $\Re{\varphi}\geq 0$ on $L$, and $\Re{\varphi^{-1}(0)}$ meets $K$ only in a point, then $K\cup L$ is polynomially convex.
\end{lemme}

From this lemma, it follows that the union of two closed caps in $\overline{\B}_N$ which intersect at a single point is polynomially convex.

\begin{proof}[Proof of Theorem \ref{last-thm-sim}]By the Baire Category Theorem and a diagonal argument, it is enough to prove the theorem for $(x_i)_i$ finite. Let us then fix $\{x_i:\,i=1,\ldots,m\}$ in $b\B_N$. Let us set a sequence $(h^j_1,\ldots h^j_m)$ dense in $\mathcal{C}(bL_{x_1})\times \ldots \times \mathcal{C}(bL_{x_m})$. By the one variable Mergelyan theorem, any continuous function on $bL_{x_i}$ can be uniformly approximated on any proper compact subset of $bL_{x_i}$ by polynomials on $\C^N$, so we can assume that each $h^j_i$ is a polynomial. Let us also fix a sequence $(y^s_1,\ldots,y^s_m)_s$ dense in $bL_{x_1}\times\ldots \times bL_{x_m}$ and let $(K_{1,n}^s,\ldots,K_{m,n}^s)_n$ be an exhaustion of compact subsets of $(bL_{x_1}\times\ldots \times bL_{x_m})\setminus \{y^s_1,\ldots,y_m^s\}$. Now, observe that
	\[
	\mathcal{U}=\bigcap_{j,n,s,k}\bigcup_{r\in [1-\frac{1}{k},1)}\bigcap_{i=1}^m\left\{f\in \mathcal{O}(\B_N):\,\sup_{K_{i,n}^s}|f_r(z)-h^j_i(z)|<\frac{1}{k}\right\}.
	\]
For $j,n,s,k$ fixed, the union on the right-hand side is clearly open in $\mathcal{O}(\B_N)$. To prove that it is dense, we can consider, for any $0<\rho<1$, finitely many closed caps $\Gamma_1,\ldots,\Gamma_{l_0} \subset \overline{\B}_N$ such that
\begin{enumerate}
	\item For any $1\leq l\leq l_0$, $\Gamma_l\cap \rho\overline{\B}_N=\emptyset$;
	\item For any $1\leq l'\leq l_0$, $\Gamma_{l'}\cap \bigcup_{l\neq l'}\Gamma_l$ contains at most one point;
	\item $K_{1,n}^s\cup \ldots \cup K_{m,n}^s\subset \bigcup _{l=1}^{l_0}\Gamma_l$.
\end{enumerate}
To finish, we use the fact that $\bigcup _{l=1}^{l_0}\Gamma_l \cup \rho\overline{\B}_N$ is polynomially convex, that we deduce from Corollary \ref{lemma-Kallin-last}.
\end{proof}

We shall mention that a function in $\mathcal{O}(\D)$ which satisfies property ($\mathcal{P}$) cannot have an arbitrary slow radial growth (\cite[Corollary 2.11]{charp}). As a consequence, and in contrast with Theorem \ref{thm-strength-ball}, a function in the class $\mathcal{U}$ cannot grow arbitrarily slowly either.

Finally, using the formalism and the ideas considered in this section, we shall add that Theorem \ref{last-thm-sim} can be stated for strictly pseudoconvex domains, up to adequate modifications.

\section*{acknowledgement}The authors are grateful to the referee for her/his careful reading of the manuscript.

\end{document}